\documentclass[11pt,reqno]{amsart}
\usepackage[margin=1.12in]{geometry}
\usepackage{amsmath,amssymb,amsthm,mathtools}
\usepackage{aliascnt}
\usepackage{microtype}
\usepackage{enumitem}
\usepackage{parskip}
\usepackage[hidelinks,pagebackref=true]{hyperref}\hypersetup{
	colorlinks,
	citecolor=blue,
	linkcolor=blue,
	urlcolor=blue}
\usepackage[nameinlink,capitalise]{cleveref}
\usepackage{comment}
\usepackage{cite}
\newtheorem{theorem}{Theorem}[section]
\newaliascnt{proposition}{theorem}
\newtheorem{proposition}[proposition]{Proposition}
\aliascntresetthe{proposition}
\newaliascnt{lemma}{theorem}
\newtheorem{lemma}[lemma]{Lemma}
\aliascntresetthe{lemma}
\newaliascnt{corollary}{theorem}
\newtheorem{corollary}[corollary]{Corollary}
\aliascntresetthe{corollary}
\theoremstyle{remark}
\newaliascnt{remark}{theorem}
\newtheorem{remark}[remark]{Remark}
\aliascntresetthe{remark}
\crefname{theorem}{Theorem}{Theorems}
\crefname{proposition}{Proposition}{Propositions}
\crefname{lemma}{Lemma}{Lemmas}
\crefname{corollary}{Corollary}{Corollaries}
\crefname{remark}{Remark}{Remarks}

\DeclareMathOperator{\Isom}{Isom}
\DeclareMathOperator{\Stab}{Stab}

\DeclareMathOperator{\sys}{sys}

\DeclareMathOperator{\arccosh}{arcosh}

\newcommand{\tn}{\textnormal}

\newcommand{\Z}{\mathbb{Z}}

\newcommand{\hs}{\mathbb{H}^{d}}

\newcommand{\mkv}{Markovi\'c }

\newcommand{\gt}{\mathcal{G}_2}

\newcommand{\mg}{\mathcal{G}}

\def\burl#1{\url{#1}}

\title[Growth of simple closed geodesics]{Growth of Simple Closed Geodesics in Finite-Volume Hyperbolic Manifolds}
\date{}
\author{Xiaolong Hans Han}
\address{Center for Mathematics and Interdisciplinary Sciences, Fudan University, Shanghai 200433, China}
\address{Shanghai Institute for Mathematics and Interdisciplinary Sciences (SIMIS), Shanghai 200433, China}
\email{xhh@simis.cn}
\author{Bohan Yang}
\address{Shanghai Institute for Mathematics and Interdisciplinary Sciences (SIMIS), Shanghai 200433, China}
\address{Research Institute of Intelligent Complex Systems, Fudan University, Shanghai
200433, China}
\email{bhyang@simis.cn}

\begin{document}
\begin{abstract}
Let $M$ be a finite-volume hyperbolic $d$-manifold with $d\ge3$.
We prove that the number of primitive nonsimple closed geodesics has
exponential growth rate strictly smaller than $d-1$. Consequently,
asymptotically almost every primitive closed geodesic in $M$ is simple. In contrast, we show that on an arithmetic hyperbolic manifold of type~I,
the unit vectors tangent to nonsimple closed geodesics are dense in the
unit tangent bundle.
\end{abstract}
\maketitle

\section{Introduction}
A closed geodesic is \emph{simple} if its geometric image is an embedded
circle.  In [\citenum{K3}, Problem~3.16], A. Reid asks whether every finite-volume hyperbolic $3$-manifold contains infinitely many simple closed geodesics.
T. Chinburg and Reid \cite{CR} construct infinitely many pairwise
noncommensurable closed hyperbolic $3$-manifolds all of whose closed
geodesics are simple.

Throughout, we count closed geodesics by their unoriented geometric
images. For a hyperbolic manifold $M$, let
\begin{align*}
N_M(L)
&=
\#\{\gamma:\gamma\text{ is a primitive closed geodesic},\ \ell(\gamma)\le L\}, \\
N_M^{\mathrm{simp}}(L)
&=
\#\{\gamma:\gamma\text{ is a primitive simple closed geodesic},\ \ell(\gamma)\le L\},
\end{align*}
and set $N_M^{\mathrm{ns}}(L)=N_M(L)-N_M^{\mathrm{simp}}(L)$.

\begin{theorem}\label{thm:main}
Let $M$ be a finite-volume hyperbolic $d$-manifold with $d\ge3$.
There exists a constant $\delta_d>0$, depending only on $d$, such that
$$
N_M^{\mathrm{ns}}(L)
\ll_M
e^{(d-1-\delta_d)L}.
$$
Consequently,
$$
\frac{N_M^{\mathrm{ns}}(L)}{N_M(L)}
\ll_M
Le^{-\delta_d L},
$$
and in particular
$$
\frac{N_M^{\mathrm{simp}}(L)}{N_M(L)}
\longrightarrow 1
\qquad\text{as }L\to\infty.
$$
\end{theorem}

Independently, using different techniques, Q. Luo--V. \mkv
\cite{lmSimpleClosedGeodesics} obtain a stronger local geometric conclusion: for a closed hyperbolic $d$-manifold with $d\ge3$ and a fixed $\kappa>\frac{3}{d-2}$, a random closed geodesic of length close to $r$ has an embedded collar of width $r^{-\kappa}$ and hence is simple.
 The prime geodesic theorem, first proved by H. Huber \cite{Huber} for
compact hyperbolic surfaces, was established by G. Margulis
\cite{Margulis} in the general compact case and by R. Gangolli--G. Warner
\cite{GangolliWarner} in the finite-volume noncompact case. It gives
$$
N_M(L)\asymp_M \frac{e^{(d-1)L}}{L}.
$$
By M. Mirzakhani [\citenum{mmGrowthSimpleClosedGeodesics}], for $S$ a complete hyperbolic surface of genus $g$ with $n$ cusps, $$N_S^{\mathrm{simp}}(L)\sim \frac{B_S}{b_{g,n}}\cdot L^{6g-6+2n}, $$ where $B_S$ depends on the hyperbolic metric and $b_{g,n}$ depends
only on $g$ and $n$. 
\Cref{thm:main} was motivated by work on the topological properties of
random closed geodesics on hyperbolic surfaces, for example by
F. Arana-Herrera \cite{ahfEffectiveFilling} and Y. Wu--Y. Xue
\cite{wxPrimeGeodesicTheorem}; see also the references therein.

S. Kuhlmann \cite{KuhlmannCusped,KuhlmannClosed} proves that every cusped orientable hyperbolic $3$-manifold contains infinitely many simple closed geodesics and obtains a sufficient criterion for the closed case; F. Xia \cite{Xia} removes the orientability hypothesis in the cusped case. C. Adams, J. Hass, and P. Scott \cite{AHS} show that the Fuchsian
quotient associated with the thrice-punctured sphere is the only
complete non-elementary orientable hyperbolic $3$-manifold containing
no simple closed geodesic. For comparison, H.-B. Rademacher
\cite{Rademacher} proves that, for a generic Riemannian metric on a
compact manifold of dimension at least three, every prime closed
geodesic is simple. L. Dever--D. Mili\'cevi\'c
\cite{DeverMilicevic} obtain effective refinements of geodesic
growth for compact orientable hyperbolic $3$-manifolds.
\begin{theorem}\label{thm:dense-nonsimple}
Let $M$ be an arithmetic hyperbolic $d$-manifold of type~I with
$d\ge2$. Then the unit vectors tangent to nonsimple closed geodesics
are dense in $T^1M$.
\end{theorem}

For a finite-area hyperbolic surface $S$, J. Birman--C. Series \cite{bsGeodesicsBoundedIntersection}  prove that the unit vectors tangent to simple complete geodesics form a nowhere dense subset of $T^1S$, while
periodic vectors are dense.
We recall the definition of an arithmetic hyperbolic manifold of type~I in
\cref{sec:prelim}.

We outline the proofs. A self-intersection of the closed geodesic
represented by a loxodromic element $g\in\Gamma$ lifts to an
intersection $A_g\cap hA_g\ne\varnothing$ for some
$h\notin\Stab_\Gamma(A_g)$. In dimension three, write the endpoints
of $A_g$ and $hA_g$ as $(a_1,b_1)$ and $(a_2,b_2)$. Intersecting axes
are coplanar, and hence their cross ratio $(a_1,a_2;b_1,b_2)$ is real; see
\cite[Lemma~5.3.11]{mrAritheoremeticHyperbolic3Manifolds}. This gives
explicit polynomial conditions in the coefficients of $h$ and $g$ whose vanishing corresponds to the coplanarity of $A_g$ and $hA_g$. 
In arbitrary
dimension, we use a Lorentzian incidence construction to show that an
intersection forces
$$
P_h(g)=0,
$$
where $P_h$ is a nonzero polynomial of uniformly bounded degree.

The stabilizers of the polynomial subspaces arising from the self-intersection conditions are proper algebraic subgroups. The subgroup theorem of C. Boubel and A. Zeghib \cite{BoubelZeghib}, together with a uniform orbit-counting estimate, gives $O(e^{(d-2)R})$ elements of $\Gamma$ at most $R$ in each relevant stabilizer coset. Comparing this bound with the hyperbolic lattice-point asymptotic for lattice balls shows that the proportion of lattice points lying in any such coset decays exponentially. A product-escape argument then gives exponential decay for polynomial conditions on products of lattice elements.

To prove \cref{thm:main}, we count nonsimple closed geodesics directly. In the shallow regime (i.e., uniformly close to the compact core of the $M$), a self-intersection splits the geodesic into a shorter and a longer arc and imposes a bounded-degree polynomial condition. Decomposing the longer arc into a fixed number of factors lying in controlled lattice balls, we obtain an exponential saving from polynomial escape. In the deep-cusp regime, a long cusp excursion yields a parabolic factor, and the slower growth of parabolic balls gives the corresponding saving. Balancing the two estimates with a cusp cutoff proportional to $L$ yields a strict exponential growth gap for nonsimple closed geodesics.

For \Cref{thm:dense-nonsimple}, we start with an arithmetic hyperbolic $3$-manifold $M$ of type~I. By S. Mozes--N. Shah \cite{msErgodicInvariantMeasure}, it
contains a sequence of noncommensurable totally geodesic properly
immersed surfaces $S_i$ whose Grassmann bundles $\gt S_i$ become dense
in $\gt M$. On each $S_i$, Birman--Series implies that the unit vectors tangent to nonsimple closed geodesics are dense
in $T^1S_i$. Thus by taking the union, we prove that nonsimple closed geodesics have dense tangent unit vectors in $T^1M$. The higher dimensional case follows naturally by induction. 

The paper is organized as follows. The polynomial intersection
criterion is developed in \cref{sec:collision-reduction}, while
\cref{sec:escape} establishes the polynomial escape estimate for
products of lattice elements.
\cref{sec:nonsimple-count} combines the polynomial escape estimate for lattice-ball products with a decomposition according to cusp depth to prove \cref{thm:main}. Finally,
\cref{sec:dense-nonsimple} proves the density of nonsimple closed geodesics for arithmetic manifolds of type~I.
\tableofcontents

\section{Preliminaries}\label{sec:prelim}
Let
$$
M=\Gamma\backslash\hs
$$
be a complete finite-volume hyperbolic $d$-manifold with $d\ge3$. Set
$$
\widehat G=\operatorname{O}^{+}(d,1),
\qquad
G=\operatorname{SO}^{+}(d,1),
$$
so that $\widehat G=\Isom(\hs)$ and
$G=\Isom^+(\hs)$ in the hyperboloid model.
Let
$$
\pi_M\colon \hs\longrightarrow M
$$
be the covering map, fix the standard basepoint $o\in\hs$, and set
$$
\Gamma^+=\Gamma\cap\Isom^+(\hs).
$$
Then $\Gamma^+$ is a lattice in $G$. For a loxodromic
$g\in\Gamma$, write $A_g$ for its axis and $\ell(g)$ for its
translation length. We denote by $\sys(M)>0$ the length of a shortest
closed geodesic in $M$.

Fix pairwise disjoint cusp neighborhoods whose complement
$M_{\mathrm{thick}}\subset M$ is compact and contains $\pi_M(o)$; moreover, we assume that any two distinct cusp neighborhoods are at distance at least
$2$. Let
$\mathcal B$ be the corresponding $\Gamma$-invariant collection of
horoballs in $\hs$, and set
$$
\widetilde M_{\mathrm{thick}}=\pi_M^{-1}(M_{\mathrm{thick}}).
$$
When $M$ is compact, we take $\mathcal B=\varnothing$ and
$M_{\mathrm{thick}}=M$.

For $x\in\hs$, define its \emph{cusp depth} by
\begin{equation}\label{eq:cusp-depth}
\operatorname{dep}(x)=d(x,\widetilde M_{\mathrm{thick}}).
\end{equation}
This function is
$\Gamma$-invariant and therefore descends to $M$.

Let $D_{\mathrm{thick}}$ denote the diameter of $M_{\mathrm{thick}}$, thus 
\begin{equation}\label{eq:coarse-orbit-density}
d(x,\Gamma o)\le \operatorname{dep}(x)+D_{\mathrm{thick}}
\qquad \tn{ for } x\in\hs.
\end{equation}

We shall also use the following uniform parabolic growth estimate.
For every maximal parabolic subgroup $P<\Gamma$ and every $R\ge0$,
\begin{equation}\label{eq:parabolic-ball-count}
\sup_{z\in\widetilde M_{\mathrm{thick}}}
\#\{p\in P:d(z,po)\le R\}
\ll_M e^{(d-1)R/2},
\end{equation}
uniformly in $P$.
Indeed, $P$ is virtually free abelian of rank at most $d-1$
\cite[Theorem~7.5.7]{rjFoundationsHyperbolicManifolds}. In the
upper half-space model with the parabolic fixed point at infinity,
points of $\widetilde M_{\mathrm{thick}}$ have uniformly bounded
height. The hyperbolic distance formula
\cite[Theorem~4.6.1]{rjFoundationsHyperbolicManifolds} therefore
shows
that the relevant translation vectors lie in a Euclidean ball of
radius $O_M(e^{R/2})$, uniformly in the center. The implied constant is
uniform because $M$ has only finitely many cusps.

\begin{lemma}\label{lem:axis-basics}
Let $\mathcal A\subset\hs$ be a geodesic whose stabilizer in the
discrete torsion-free group $\Gamma$ contains a nontrivial loxodromic
element. Then
\begin{enumerate}[label=(\roman*)]
\item $\Stab_\Gamma(\mathcal A)\cong\Z$, and every element of
$\Stab_\Gamma(\mathcal A)$ preserves the orientation of $\mathcal A$.
\item If $g\in\Gamma$ is loxodromic, then the closed geodesic
represented by $g$ is simple if and only if
$$
A_g\cap hA_g=\varnothing
$$
for every $h\in\Gamma\setminus\Stab_\Gamma(A_g)$.
\end{enumerate}
\end{lemma}
\begin{proof}
The group $\Stab_\Gamma(\mathcal A)$ preserves the pair of
endpoints of $\mathcal A$ in $\partial\hs$. Hence it is an elementary discrete subgroup containing a loxodromic
element. By
\cite[Theorem~5.5.8]{rjFoundationsHyperbolicManifolds},
it contains an infinite cyclic subgroup of finite index.
Since $\Gamma$ is torsion-free,
$\Stab_\Gamma(\mathcal A)$ is itself infinite cyclic.
Moreover, an orientation-reversing isometry of $\mathcal A$ has a fixed
point on $\mathcal A$. Point stabilizers in $\Isom(\hs)$ are compact,
so an element of the discrete group $\Gamma$ that fixes a point has
finite order. Since $\Gamma$ is torsion-free, no such nontrivial element
exists. This proves~(i).

For~(ii), the lifts of the geometric image of the closed geodesic
represented by $g$ are the geodesics $hA_g$, $h\in\Gamma$. Two such
lifts $hA_g$ and $h'A_g$ agree precisely when $h^{-1}h'$ belongs to
$\Stab_\Gamma(A_g)$. 
Thus its image is embedded if
and only if distinct lifts are disjoint.
\end{proof}

For later use, we recall the definition of an arithmetic manifold of
type~I; see \cite[Section~2]{hjAsymptoticalRigidityFundamentalGroups},
\cite[Chapter~6]{vsDiscreteGroupsConstantCurvature}, and
\cite{mdArithmeticGroups}. Let $\mathbb K\subset\mathbb R$ be a
totally real number field with ring of integers $\mathcal O_{\mathbb
K}$. A nondegenerate quadratic form
$$
q(x)=\sum_{i,j=1}^{d+1}a_{ij}x_ix_j,
\qquad a_{ij}=a_{ji}\in\mathbb K,
$$
is \emph{admissible} if it has signature $(d,1)$ and, for every
nonidentity embedding $\iota\colon\mathbb K\to\mathbb R$, the
conjugate form
$$
q^\iota(x)=\sum_{i,j=1}^{d+1}\iota(a_{ij})x_ix_j
$$
is positive definite. The subgroup
$O'(q,\mathcal O_{\mathbb K})$ consisting of the integral isometries
of $q$ that preserve each component of
$\{x\in\mathbb R^{d+1}:q(x)<0\}$ is a lattice in
$\Isom(\hs)$. More generally, if $\mathcal L$ is an
$\mathcal O_{\mathbb K}$-lattice in $\mathbb K^{d+1}$, we write
$O'(\mathcal L,q)$ for the corresponding group of integral
isometries. A lattice commensurable with some $O'(\mathcal L,q)$ is
called an \emph{arithmetic lattice of type~I}. When such a lattice $\Gamma$ is
torsion-free, the quotient $\Gamma\backslash\hs$ is called an
\emph{arithmetic hyperbolic manifold of type~I}.

\section{Polynomials for self-intersections}
\label{sec:collision-reduction}
The work of Chinburg and Reid \cite{CR} motivates us to associate
concrete polynomial vanishing conditions with self-intersections of
closed geodesics in hyperbolic $3$-manifolds. We generalize this idea
using a Lorentzian incidence condition that works uniformly in every
dimension, although the resulting polynomials are implicit.

Let $\mathbb R^{d,1}$ be equipped with the Lorentzian form
$$
\langle x,y\rangle=x_1y_1+\cdots+x_dy_d-x_{d+1}y_{d+1}.
$$
Write $J=\operatorname{diag}(1,\ldots,1,-1)\in
M_{d+1}(\mathbb R)$, so that $\langle x,y\rangle=x^TJy$, and identify
$\hs$ with $\{x\in\mathbb R^{d,1}:\langle x,x\rangle=-1, x_{d+1}>0 \}$. For a linear map
$g$, we use the convention
$$
(\wedge^2g)(a\wedge b)=ga\wedge gb.
$$
Every geodesic $A\subset\hs$ determines a unique Lorentzian
two-plane $E_A\subset\mathbb R^{d,1}$.
\begin{lemma}\label{lem:collision-polynomial}
There is an integer $D_0=D_0(d)\ge1$ with the following property. For
every $h\in\Gamma\setminus\{e\}$, there is a real polynomial $P_h$ in
the matrix entries, of degree at most $D_0$, which is not
identically zero on $G$ and satisfies, for every loxodromic
$g\in\Gamma^+$,
$$
A_g\cap hA_g\ne\varnothing
\quad\Longrightarrow\quad
P_h(g)=0.
$$
\end{lemma}

\begin{proof}
For $x\in\mathbb R^{d,1}$, set
$$
w_h(x)=x\wedge h^{-1}x.
$$
Since $h\in\widehat G$, we have $h^{-1}=J^{-1}h^TJ$. Thus the
entries of $h^{-1}$ depend linearly on those of $h$. Define
$$
\Delta_h(x)
:=
\langle x,x\rangle\langle h^{-1}x,h^{-1}x\rangle
-
\langle x,h^{-1}x\rangle^2.
$$
Consider the semialgebraic incidence set
$$
\mathcal I_h
=
\left\{
(g,x)\in G\times\hs:
\begin{array}{l}
(\wedge^2g)w_h(x)=w_h(x)
\end{array}
\right\}.
$$
Since $\Gamma$ is discrete and torsion-free, $h$ has no fixed point
in $\hs$. Moreover, \cite[(3.2.2)]{rjFoundationsHyperbolicManifolds}
gives
$$
\langle x,h^{-1}x\rangle
=
-\cosh d(x,h^{-1}x).
$$
Since $\langle x,x\rangle=\langle h^{-1}x,h^{-1}x\rangle=-1$,
$$
\Delta_h(x)
=
-\sinh^2 d(x,h^{-1}x)<0.
$$

Thus $w_h(x)$ is a nonzero simple bivector spanning a nondegenerate
Lorentzian two-plane. For each such $x$, the condition
$$
(\wedge^2 g)w_h(x)=w_h(x)
$$
says that $g$ stabilizes the oriented two-plane spanned by $x$ and
$h^{-1}x$ and has determinant one on it. The stabilizer of this
bivector is a smooth algebraic subgroup of $\operatorname{SO}(d,1)$,
conjugate to
$$
\operatorname{SO}(1,1)\times\operatorname{SO}(d-1),
$$
up to connected components.

For fixed $x\in\hs$, the set of $g\in G$ satisfying
$
(\wedge^2g)w_h(x)=w_h(x)
$
has dimension
$$
1+\frac{(d-1)(d-2)}2.
$$
Since $x$ ranges over the $d$-dimensional hyperboloid,
$$
\dim\mathcal I_h
\le
1+\frac{(d-1)(d-2)}2+d
=
\dim G-(d-2).
$$
Let
$$
\mathcal X_h
:=
\{g\in G:(g,x)\in\mathcal I_h\text{ for some }x\}.
$$
The Tarski--Seidenberg theorem
\cite[Proposition~2.2.7]{BochnakCosteRoy} shows that $\mathcal X_h$ is
semialgebraic. Since the projection does not increase semialgebraic dimension
\cite[Theorem~2.8.8]{BochnakCosteRoy},
$$
\dim\mathcal X_h
\le
\dim\mathcal I_h\le \dim G-(d-2)
<\dim G.
$$


Suppose that $g\in G$ is loxodromic and
$A_g\cap hA_g\ne\varnothing$. Choose $x\in A_g\cap hA_g$. Then
$x\ne h^{-1}x$, and both points lie in $A_g$; hence
$$
E_{A_g}=\operatorname{span}\{x,h^{-1}x\}.
$$
Since $g$ preserves $E_{A_g}$, we claim that
$g|_{E_{A_g}}\in\operatorname{SO}(1,1)$. Let $v^+,v^-\in E_{A_g}$
be two linearly independent eigenvectors of $g$ with eigenvalues $\lambda^+$ and
$\lambda^-$.  Since $g$ is loxodromic,
$(\lambda^\pm)^2\ne1$. As $g$ preserves the Lorentzian form,
$$
\langle v^\pm,v^\pm\rangle
=
\langle gv^\pm,gv^\pm\rangle
=
(\lambda^\pm)^2\langle v^\pm,v^\pm\rangle,
$$
implying that $v^\pm$ are lightlike. Since $v^+$ and $v^-$ are nonzero non-spacelike vectors whose last coordinates are positive, by [\citenum{rjFoundationsHyperbolicManifolds}, Theorem 3.1.1], $\langle v^+,v^-\rangle$ is not $0$. As $g$ preserves the Lorentzian
form,
$$
\langle v^+,v^-\rangle
=
\langle gv^+,gv^-\rangle
=
\lambda^+\lambda^-\langle v^+,v^-\rangle,
$$
so $\lambda^+\lambda^-=1$. Thus the restriction of $g$ to
$E_{A_g}$ has determinant one and belongs to $\operatorname{SO}(1,1)$.
In particular, it fixes the oriented area form of this plane, and hence
\begin{equation}\label{eq:wedge2g}
(\wedge^2g)(x\wedge h^{-1}x)
=
x\wedge h^{-1}x.
\end{equation}
Therefore,
$(g,x)\in\mathcal I_h$, and $g\in\mathcal X_h$.

It remains to find a single polynomial, of degree bounded independently
of $h$, whose zero set contains $\mathcal X_h$. 
For fixed $h$,
membership of $g$ in $\mathcal X_h$ is expressed by an existential  semialgebraic formula in the quantified variable $x$, with $g$ as the free variable. 
This only involves a bounded number 
of polynomials of degree bounded in terms of $d$, with $h$ entering only
through their coefficients. Applying quantitative
quantifier elimination in the variables $x$
\cite[Theorem~2.27]{BasuRealAlgebraicGeometry}, we obtain a
quantifier-free description of $\mathcal X_h$ by sign conditions on
polynomials
$$
q_{h,1},\ldots,q_{h,N},
$$
where $N$ and the degrees of the $q_{h,i}$ are bounded in terms of
$d$ alone.

If $\mathcal X_h=\varnothing$, take $P_h=1$. Otherwise, retain only
those $q_{h,i}$ that are not identically zero on $G$. For every
$g\in\mathcal X_h$, at least one of the remaining polynomials vanishes
at $g$. Indeed, if they were all nonzero at some $g\in\mathcal X_h$,
their signs would be constant in a neighborhood of $g$ in $G$, forcing that
neighborhood to lie in $\mathcal X_h$, contrary to
$\dim\mathcal X_h<\dim G$. Hence
$$
P_h\coloneqq\prod_i q_{h,i}
$$
vanishes on $\mathcal X_h$. Since $G$ is connected and each retained
$q_{h,i}$ is a nonzero real-analytic function on $G$, the product is not identically zero on
$G$, and the uniform bounds above give
$$
\deg P_h\le D_0(d)
$$
for some $D_0(d)$ independent of $h$. Enlarging $D_0(d)$ if necessary,
we may assume that $D_0(d)\ge1$.

\end{proof}

\section{Polynomial escape}\label{sec:escape}

The intersection criterion from the preceding section reduces a
self-intersection to the vanishing of a polynomial of uniformly bounded
degree. We first obtain uniform growth bounds for the stabilizers of
the corresponding polynomial subspaces and then derive escape
estimates for products of lattice elements.

Let $D_0(d)$ be the uniform degree bound in
\cref{lem:collision-polynomial}, and let $\mathcal P$ be the finite-dimensional real vector space obtained
by restricting the real polynomials of
total degree at most $D_0(d)$ in the ambient matrix coordinates to $\operatorname{SO}(d,1)$. Since
$G=\operatorname{SO}^+(d,1)$ is Zariski dense in
$\operatorname{SO}(d,1)$, every nonzero element of $\mathcal P$ does not vanish identically on $G$. For $g\in\operatorname{SO}(d,1)$, define
$$ 
R_g\colon\mathcal P\longrightarrow\mathcal P,
\qquad
(R_gf)(x)=f(xg).
$$ 
Then $R_gR_h=R_{gh}$, and right translation preserves $\mathcal P$.

For a subspace $W\le\mathcal P$, write
$$ 
Z(W)=\{x\in G:f(x)=0\text{ for every }f\in W\},
\qquad
H_W=\{g\in\operatorname{SO}(d,1):R_gW=W\}.
$$ 

\begin{lemma}\label{lem:stabilizer-proper}
If $W\le\mathcal P$ is nonzero and $Z(W)\ne\varnothing$, then $H_W$
is a proper real algebraic subgroup of $\operatorname{SO}(d,1)$.
\end{lemma}

\begin{proof}
For every $f\in\mathcal P$, the coefficients of $R_gf$ depend
polynomially on the matrix entries of $g$. Thus
$$ 
g\longmapsto R_g
$$ 
is an algebraic representation, and $H_W$ is the inverse image of the
stabilizer of $W$ in $\operatorname{GL}(\mathcal P)$.

Suppose that $H_W=\operatorname{SO}(d,1)$, and choose $x_0\in Z(W)$.
For every $f\in W$ and $g\in\operatorname{SO}(d,1)$, we have
$R_gf\in W$, and thus
$$ 
f(x_0g)=(R_gf)(x_0)=0.
$$ 
Since $x_0\operatorname{SO}(d,1)=\operatorname{SO}(d,1)$, every
$f\in W$ vanishes identically, contradicting $W\ne0$.
\end{proof}

\begin{lemma}\label{lem:subgroup-coset-count}
For every $a\in\Gamma^+$, every
nonzero $W\le\mathcal P$ with $Z(W)\ne\varnothing$, and every $R\ge0$,
\begin{equation}\label{eq:actual-coset-count}
\#\{\lambda\in\Gamma^+\cap aH_W:d(o,\lambda o)\le R\}
\ll_M e^{(d-2)R},
\end{equation}
uniformly in $a$, $W$, and $R$.
\end{lemma}

\begin{proof}
Since $a\in\Gamma^+$,
\begin{equation}\label{eq:actual-coset}
\Gamma^+\cap aH_W=a(\Gamma^+\cap H_W).
\end{equation}
Put
$$ 
\Gamma_W=\Gamma^+\cap H_W<\Gamma^+.
$$ 
If $\lambda=a\lambda_0$ with $\lambda_0\in\Gamma_W$, then
$$ 
d(o,\lambda o)=d(a^{-1}o,\lambda_0o).
$$ 
Therefore, it suffices to count the $\Gamma_W$-orbit in a ball centered
at $a^{-1}o$.

Let $\mathcal G_W$ be the real Zariski closure of $\Gamma_W$ in
$\operatorname{SO}(d,1)$. By \cref{lem:stabilizer-proper},
$\mathcal G_W$ is proper.  We use the following geometric dichotomy.

If the identity component $\mathcal G_W^\circ$ is compact, then $\mathcal G_W$ is compact,
since a real algebraic group has finitely many connected components.
Hence the discrete subgroup $\Gamma_W\le \mathcal G_W$ is finite.
Since $\Gamma_W$ is torsion-free, it follows that $\Gamma_W$ is
trivial.
Suppose that $\mathcal G_W^\circ$ is noncompact. If it fixes a point of
$\partial\mathbb H^d$, the normality of $\mathcal G_W^\circ$ in
$\mathcal G_W$ implies that the $\mathcal G_W$-orbit of this point is
finite, so $\Gamma_W$ is elementary. Otherwise, by the geometric Lie
subgroup theorem of Boubel--Zeghib
\cite[Theorem~1.1]{BoubelZeghib}, $\mathcal G_W^\circ$ preserves a
totally geodesic subspace
$$ 
\Pi\simeq\mathbb H^j
$$ 
and contains the identity component of its isometry group. Its limit
set is therefore $\partial\Pi$. The normality of
$\mathcal G_W^\circ$ in $\mathcal G_W$ implies that
$\mathcal G_W$ preserves $\partial\Pi$, hence $\Pi$. Since
$\mathcal G_W$ is proper, $j<d$. Thus $\Gamma_W$ is either elementary
or preserves a totally geodesic subspace
$\Pi\simeq\mathbb H^j$ with $2\le j\le d-1$.

We now prove \eqref{eq:actual-coset-count}. If $\Gamma_W$ is elementary
and contains a loxodromic element, then it is infinite cyclic
\cite[Proposition~5.1.4]{bmIntroductionGeometricTopology}. Since the
lengths of closed geodesics in $M$ are at least $\sys(M)$, every
radius-$R$ ball contains $O_M(R+1)$ points of such an orbit.

Suppose that $\Gamma_W$ is parabolic of rank at most $d-1$. The
$\Gamma$-invariant thick part $\widetilde M_{\mathrm{thick}}$ contains
both $o$ and $a^{-1}o$. Since $\Gamma_W$ is contained in a maximal
parabolic subgroup of $\Gamma$, \eqref{eq:parabolic-ball-count} gives
$$ 
\#\{\lambda\in\Gamma_W:d(a^{-1}o,\lambda o)\le R\}
\ll_M e^{(d-1)R/2}
\ll_M e^{(d-2)R},
$$ 
where the last inequality uses $d\ge3$.

It remains to consider the case in which $\Gamma_W$ preserves
$\Pi\simeq\mathbb H^j$ with $2\le j\le d-1$. Put
$$ 
y=\operatorname{pr}_\Pi(o),
\qquad
y_a=\operatorname{pr}_\Pi(a^{-1}o),
\qquad
t_0=\operatorname{dep}(y).
$$ 
Since $\Gamma_W$ preserves $\Pi$ and permutes the cusp horoballs, all
points $\lambda y$ have depth $t_0$. The orbit map
$\lambda\mapsto\lambda y$ is injective because its point stabilizer is
finite and $\Gamma_W$ is torsion-free.

We first record a local estimate which is uniform in the center. For
every $z\in\Pi$ and $0<\varrho\le1$,
\begin{equation}\label{eq:local-plane-packing}
\#\{\lambda\in\Gamma_W:d_\Pi(z,\lambda y)\le\varrho\}
\ll_M(1+e^{t_0}\varrho)^{j-1}.
\end{equation}
Indeed, suppose first that $t_0>1$, and let
$\mathcal H_0\in\mathcal B$ be the unique cusp horoball containing
$y$. Any two orbit points $\lambda_1y$ and $\lambda_2y$ in
$B_\Pi(z,\varrho)$ are at distance at most $2\varrho\le2$.
Since
$$
\lambda_i y\in \lambda_i\mathcal H_0
\qquad (i=1,2),
$$
and two points of depth $t_0>1$ lying in distinct cusp horoballs have
distance greater than $2$, we must have
$$
\lambda_1\mathcal H_0=\lambda_2\mathcal H_0.
$$
Hence
$$
\lambda_2^{-1}\lambda_1
\in \operatorname{Stab}_\Gamma(\mathcal H_0).
$$
After applying $\lambda_2^{-1}$, the two points become $y$ and
$\lambda_2^{-1}\lambda_1y$ on the horosphere of depth $t_0$ in
$\mathcal H_0$. The induced metric on this horosphere is the Euclidean
metric scaled by $e^{-t_0}$. Moreover, the subgroup of the cusp
stabilizer preserving $\Pi$ is virtually free abelian of rank at most
$j-1$. Since $M$ has only finitely many cusp orbits, the corresponding
Euclidean lattice counting is uniform and gives \eqref{eq:local-plane-packing}.

If $t_0\le1$ and the ball contains an orbit point, then
$\operatorname{dep}(z)$ is at most $2$. 
The set of points in $M$ of depth at most
$2$ is compact, so the same estimate follows from uniform discreteness
on this compact set.

For $\varrho\ge1$, a radius-$\varrho$ ball in
$\Pi\simeq\mathbb H^j$ can be covered by
$O(e^{(j-1)\varrho})$ unit balls. Applying
\eqref{eq:local-plane-packing} to each of them gives
\begin{equation}\label{eq:global-plane-packing}
\#\{\lambda\in\Gamma_W:d_\Pi(y_a,\lambda y)\le\varrho\}
\ll_M e^{(j-1)(\varrho+t_0)}.
\end{equation}
Set
$$ 
\tau_o=d(o,\Pi),
\qquad
\tau_a=d(a^{-1}o,\Pi),
\qquad
s_\lambda=d_\Pi(y_a,\lambda y).
$$ 
The normal decomposition in the hyperboloid model
\cite[Section~3.2]{rjFoundationsHyperbolicManifolds} gives
$$ 
o=\cosh\tau_o\,y+\sinh\tau_o\,u_o,
\qquad
a^{-1}o=\cosh\tau_a\,y_a+\sinh\tau_a\,u_a,
$$ 
where $u_o,u_a$ are unit spacelike vectors orthogonal to $\Pi$. Since
$\lambda\Pi=\Pi$,
$$ 
\lambda o=\cosh\tau_o\,\lambda y+\sinh\tau_o\,\lambda u_o.
$$ 
The orthogonal complement $\Pi^\perp$ is positive definite, so
$\langle u_a,\lambda u_o\rangle\le1$. Together with
$\langle y_a,\lambda y\rangle=-\cosh s_\lambda$, this gives
$$ 
\begin{aligned}
\cosh d(a^{-1}o,\lambda o)
&\ge
\cosh\tau_a\cosh\tau_o\cosh s_\lambda
-\sinh\tau_a\sinh\tau_o\\
&=
\cosh(\tau_a-\tau_o)
+\cosh\tau_a\cosh\tau_o(\cosh s_\lambda-1).
\end{aligned}
$$ 
Hence $d(a^{-1}o,\lambda o)\le R$ implies
\begin{equation}\label{eq:projection-estimate}
|\tau_a-\tau_o|\le R,
\qquad
\cosh s_\lambda-1\le4e^{R-\tau_a-\tau_o}.
\end{equation}
Since $o\in\widetilde M_{\mathrm{thick}}$, we have
$$ 
t_0\le d(y,o)=\tau_o.
$$ 

Suppose first that
$$ 
\tau_a+\tau_o
\le R+\log\frac{4}{\cosh1-1}.
$$ 
Then \eqref{eq:projection-estimate} gives
$$ 
s_\lambda\le R-\tau_a-\tau_o+C_1,
\tn{ where }
C_1=1+\log\frac{4}{\cosh1-1}.
$$ 
Thus $s_\lambda+t_0\le R+C_1$. If the set being counted is nonempty,
then also $t_0\le R+C_1$. Taking
$$ 
\varrho=\max\{1,R-t_0+C_1\}
$$ 
in \eqref{eq:global-plane-packing} yields
$$ 
\#\{\lambda\in\Gamma_W:d(a^{-1}o,\lambda o)\le R\}
\ll_M e^{(j-1)R}.
$$ 

Suppose instead that
$$ 
\tau_a+\tau_o
>R+\log\frac{4}{\cosh1-1},
$$ 
and put
$$ 
\varrho=
\arccosh\left(1+4e^{R-\tau_a-\tau_o}\right)<1.
$$ 
Then $s_\lambda\le\varrho$. Since $t_0\le\tau_o$, $\arccosh(1+x)\le \sqrt{2x}$, and
$\tau_o-\tau_a\le R$, we have
$$ 
e^{t_0}\varrho
\ll e^{(R+\tau_o-\tau_a)/2}
\le e^R.
$$ 
Equation \eqref{eq:local-plane-packing} therefore gives
$$ 
\#\{\lambda\in\Gamma_W:d(a^{-1}o,\lambda o)\le R\}
\ll_M e^{(j-1)R}.
$$ 
Since $j\le d-1$, this proves \eqref{eq:actual-coset-count}.
\end{proof}

For a nonzero $F\in\mathcal P$, let
$$
V_F=\operatorname{span}\{R_gF:g\in G\}.
$$
Then $V_F$ is invariant under right translation, and the argument in the
proof of \cref{lem:stabilizer-proper}, applied to the $G$-action, gives
\begin{equation}\label{eq:translated-span-empty}
Z(V_F)=\varnothing.
\end{equation}

We now pass from coset counting to product escape. The following probabilistic estimate will later be applied to normalized counting measures on finite lattice balls.

Put
$$
m_0=\dim\mathcal P-1.
$$
\begin{proposition}
\label{prop:product-escape}
Let $(\Omega,\mathcal F,\mathbb P)$ be a probability space, and let
$X_1,\ldots,X_{m_0}\colon\Omega\to\Gamma^+$
be independent random variables. Suppose that, for some
$0<\beta\le1$,
\begin{equation}\label{eq:uniform-coset-concentration}
\mathbb P(X_i\in aH_W)\le\beta
\end{equation}
for every $1\le i\le m_0$, every $a\in\Gamma^+$, and every nonzero
$W\le\mathcal P$ with $Z(W)\ne\varnothing$. Then, for every nonzero
$F\in\mathcal P$,
$$ 
\mathbb P\bigl(F(X_1\cdots X_{m_0})=0\bigr)
\le C_{m_0}\beta^{2^{-m_0}},
$$ 
where $C_{m_0}$ depends only on $m_0$.
\end{proposition}

\begin{proof}
We prove the following more general statement, simultaneously for all
nonzero $F\in\mathcal P$, by induction on the codimension. Let $0\ne W\le V_F$ satisfy $Z(W)\ne\varnothing$, and put
$$ 
r=\operatorname{codim}_{V_F}W.
$$ 
Then, for any independent $\Gamma^+$-valued random variables
$X_1,\ldots,X_r$ satisfying
\eqref{eq:uniform-coset-concentration},
\begin{equation}\label{eq:codim-escape}
\mathbb P(X_1\cdots X_r\in Z(W))
\le C_r\beta^{2^{-r}},
\end{equation}
where the constants $C_r$ are chosen inductively and, after enlarging
them if necessary, are nondecreasing in $r$.
The case $r=0$ does not occur, since then $W=V_F$ and
\eqref{eq:translated-span-empty} gives $Z(W)=\varnothing$.
 Write
$$ 
X=X_1\cdots X_{r-1},
\qquad
Y=X_r,
$$ 
and let $Y'$ be an independent copy of $Y$, independent also of
$X_1,\ldots,X_{r-1}$. Set
$$ 
p_W=\mathbb P(XY\in Z(W)).
$$ 
Conditioning on $X$ and applying Cauchy--Schwarz gives
$$
\begin{aligned}
p_W^2
&=
\left(
\mathbb E\bigl[
\mathbb P(XY\in Z(W)\mid X)
\bigr]
\right)^2\\
&\le
\mathbb E\bigl[
\mathbb P(XY\in Z(W)\mid X)^2
\bigr]\\
&=
\mathbb P\bigl(
XY\in Z(W),\,
XY'\in Z(W)
\bigr).
\end{aligned}
$$
For fixed $Y,Y'$, the event on the right is
$$
X\in Z(R_YW+R_{Y'}W).
$$

If $R_YW=R_{Y'}W$, then
$$ 
R_{Y^{-1}Y'}W=W,
$$ 
and hence $Y'\in YH_W$. Conditioning on $Y$ and using
\eqref{eq:uniform-coset-concentration}, this case has probability at
most $\beta$.

Suppose now that $R_YW\ne R_{Y'}W$, and put
$$ 
W'=R_YW+R_{Y'}W.
$$

Since $R_YW$ has codimension $r$ in $V_F$ and is properly contained
in $W'$, we have
$$
r'\coloneqq\operatorname{codim}_{V_F}W'\le r-1.
$$
If $Z(W')=\varnothing$, the corresponding contribution is zero.
Otherwise $r'\ge1$, and hence $r\ge2$. Write
$$
X=X^{(0)}X^{(1)},
$$
where $X^{(1)}$ is the product of the last $r'$ variables among
$X_1,\ldots,X_{r-1}$, with the convention $X^{(0)}=e$ if $r'=r-1$.
Then
$$
X\in Z(W')
\quad\Longleftrightarrow\quad
X^{(1)}\in Z(L_{X^{(0)}}W'),
\tn{ where }
(L_af)(x)=f(ax).
$$
Since left and right translations commute, $L_{X^{(0)}}W'$ is a
subspace of $V_{L_{X^{(0)}}F}$ with codimension $r'$. Moreover,
$$
Z(L_{X^{(0)}}W')
=(X^{(0)})^{-1}Z(W')\ne\varnothing.
$$ 
Conditioning further on $X^{(0)}$ and applying the induction hypothesis
to the last $r'$ variables give
$$
\mathbb P\bigl(
X\in Z(W')\mid Y,Y',X^{(0)}
\bigr)
\le
C_{r'}\beta^{2^{-r'}},
$$
for the constant $C_{r'}$ in \eqref{eq:codim-escape}. Averaging over $X^{(0)}$, we obtain
$$
\mathbb P(X\in Z(W')\mid Y,Y')
\le
C_{r'}\beta^{2^{-r'}}
\le
C_{r-1}\beta^{2^{-(r-1)}}.
$$
The last inequality follows from $r'\le r-1$, the monotonicity of
the constants $C_r$, and
$0<\beta\le1$.

If $r=1$, then $R_YW\ne R_{Y'}W$ implies $W'=V_F$, so
$Z(W')=\varnothing$. Hence $p_W^2\le\beta$, and
\eqref{eq:codim-escape} holds with $C_1=1$. For $r\ge2$, combining the two cases gives
$$
p_W^2
\le
\beta+C_{r-1}\beta^{2^{-(r-1)}}
\le
(1+C_{r-1})\beta^{2^{-(r-1)}}.
$$
Choosing $C_r$ so that
$$
C_r\ge \max\{C_{r-1},\sqrt{1+C_{r-1}}\},
$$
and taking square roots proves \eqref{eq:codim-escape}.

We now apply this estimate to $F$. Put
$$ 
r_F=
\operatorname{codim}_{V_F}\operatorname{span}\{F\}
\le m_0.
$$ 
If $r_F=0$, then
$$ 
Z(\operatorname{span}\{F\})
=
Z(V_F)
=
\varnothing,
$$ 
so there is nothing to prove.

Assume $r_F\ge1$. Condition on the first $m_0-r_F$ variables and
denote their product by
$$
X^{(0)}=X_1\cdots X_{m_0-r_F},
$$
with the convention $X^{(0)}=e$ if $r_F=m_0$. Put
$$ 
F_0=L_{X^{(0)}}F.
$$ 
Since left translation maps $V_F$ isomorphically onto $V_{F_0}$,
$$ 
\operatorname{codim}_{V_{F_0}}
\operatorname{span}\{F_0\}
=
r_F.
$$ 
If $Z(\operatorname{span}\{F_0\})=\varnothing$, the corresponding
conditional probability is zero. Otherwise,
\eqref{eq:codim-escape}, applied to the remaining $r_F$ variables,
gives
$$
\mathbb P\bigl(
F(X_1\cdots X_{m_0})=0
\mid X^{(0)}
\bigr)
\le
C_{r_F}\beta^{2^{-r_F}}.
$$
Since $r_F\le m_0$ and $0<\beta\le1$,
$$
C_{r_F}\beta^{2^{-r_F}}
\le
C_{m_0}\beta^{2^{-m_0}}.
$$
Thus, averaging over $X^{(0)}$, we obtain
$$
\mathbb P\bigl(F(X_1\cdots X_{m_0})=0\bigr)
\le
C_{m_0}\beta^{2^{-m_0}}.
$$
\end{proof}

Assume first that $M$ is orientable, so that $\Gamma=\Gamma^+$. For
$T\ge0$, define the lattice ball
$$ 
B_T=\{g\in\Gamma:d(o,go)\le T\}.
$$ 
The hyperbolic lattice-point problem goes back to Delsarte
\cite{Delsarte1942} in dimension two. In arbitrary dimension, the
counting theorem of Eskin--McMullen
\cite[Theorem~1.4]{EskinMcMullen}
gives
$$
\#\{g\in\Gamma:d(o,g o)\le T\}
\sim
\frac{\operatorname{vol}_{\mathbb H^d}(B(o,T))}
{\operatorname{vol}(M)}.
$$
Since
$$
\operatorname{vol}_{\mathbb H^d}(B(o,T))
\asymp e^{(d-1)T},
$$
we obtain
\begin{equation}\label{eq:lattice-ball-growth}
|B_T|\asymp_M e^{(d-1)T}.
\end{equation}
Since $\Gamma=\Gamma^+$, \cref{lem:subgroup-coset-count} gives
\begin{equation}\label{eq:lattice-ball-coset}
|B_T\cap aH_W|\ll_M e^{(d-2)T}
\end{equation}
for every $a\in\Gamma$ and every nonzero $W\le\mathcal P$ with
$Z(W)\ne\varnothing$.

Combining
\eqref{eq:lattice-ball-growth} and \eqref{eq:lattice-ball-coset}, we obtain
\begin{equation}\label{eq:ball-coset-concentration}
\sup_{\substack{a\in\Gamma,\;0\ne W\le\mathcal P\\ Z(W)\ne\varnothing}}
\frac{|B_T\cap aH_W|}{|B_T|}
\ll_M e^{-T}.
\end{equation}

Recall $m_0=\dim\mathcal P-1$. Put
$$
\eta_{\mathcal P}
=
\frac{1}{m_0\,2^{m_0}}.
$$
\begin{corollary}\label{cor:ball-product-escape}
Let $F\in\mathcal P$ be nonzero and let $R$ be nonnegative. For
$1\le i\le m_0$, let $R_i\ge R/m_0$, and let
$X_1,\ldots,X_{m_0}$ be independent random variables such that
$X_i$ is distributed according to normalized counting measure on
$B_{R_i}$. Then
\begin{equation}\label{eq:ball-product-escape}
\mathbb P\bigl(F(X_1\cdots X_{m_0})=0\bigr)
\le
C_M e^{-\eta_{\mathcal P}R}.
\end{equation}
Equivalently,
\begin{equation}\label{eq:ball-product-count}
\#\left\{
(x_1,\ldots,x_{m_0})\in\prod_{i=1}^{m_0}B_{R_i}:
F(x_1\cdots x_{m_0})=0
\right\}
\le
C_M e^{-\eta_{\mathcal P}R}
\prod_{i=1}^{m_0}|B_{R_i}|.
\end{equation}
\end{corollary}
\begin{proof}
Fix $1\le i\le m_0$. Since $X_i$ is distributed according to normalized
counting measure on $B_{R_i}$, for every $a\in\Gamma$ and every nonzero
subspace $W\le\mathcal P$ with $Z(W)\ne\varnothing$,
$$
\mathbb P(X_i\in aH_W)
=
\frac{|B_{R_i}\cap aH_W|}{|B_{R_i}|}.
$$
By \eqref{eq:ball-coset-concentration},
$$
\frac{|B_{R_i}\cap aH_W|}{|B_{R_i}|}
\le
C_Me^{-R_i}.
$$
Since $R_i\ge R/m_0$, it follows that
$$
\mathbb P(X_i\in aH_W)
\le
C_Me^{-R/m_0}.
$$
Thus, uniformly in $i$, $a$, and $W$, the concentration hypothesis
\eqref{eq:uniform-coset-concentration} of
\cref{prop:product-escape} holds with
$$
\beta=C_Me^{-R/m_0},
$$
provided $R$ is sufficiently large that $\beta\le1$. Applying \cref{prop:product-escape} gives
$$
\begin{aligned}
\mathbb P\bigl(F(X_1\cdots X_{m_0})=0\bigr)
\le
C_{m_0}\beta^{2^{-m_0}}
=
C_{m_0}\left(C_Me^{-R/m_0}\right)^{2^{-m_0}}
=
C_{m_0}C_M^{2^{-m_0}}
e^{- R/(m_0 2^{m_0})}.
\end{aligned}
$$
Since $m_0$ is fixed, the factor
$C_{m_0}C_M^{2^{-m_0}}$ may be absorbed into the constant $C_M$.
Recalling that
$
\eta_{\mathcal P}
=
\frac{1}{m_0\,2^{m_0}},
$
we obtain
$$
\mathbb P\bigl(F(X_1\cdots X_{m_0})=0\bigr)
\le
C_Me^{-\eta_{\mathcal P}R}
$$
for all sufficiently large $R$.

For bounded $R$, the estimate follows from the trivial bound
$$
\mathbb P\bigl(F(X_1\cdots X_{m_0})=0\bigr)\le1,
$$
after increasing $C_M$ if necessary. Thus
\eqref{eq:ball-product-escape} holds for every $R\ge0$.

Finally, since the random variables $X_1,\ldots,X_{m_0}$ are independent
and each $X_i$ is uniform on $B_{R_i}$, the tuple
$(X_1,\ldots,X_{m_0})$ is uniformly distributed on
$$
\prod_{i=1}^{m_0}B_{R_i}.
$$
Therefore
$$
\mathbb P\bigl(F(X_1\cdots X_{m_0})=0\bigr)
=
\frac{
\#\left\{
(x_1,\ldots,x_{m_0})\in\prod_{i=1}^{m_0}B_{R_i}:
F(x_1\cdots x_{m_0})=0
\right\}
}{
\prod_{i=1}^{m_0}|B_{R_i}|
}.
$$
Combining this identity with \eqref{eq:ball-product-escape} gives
\eqref{eq:ball-product-count}.
\end{proof}

\section{Counting nonsimple closed geodesics}
\label{sec:nonsimple-count}
We divide the counting of nonsimple closed geodesics according to cusp
depth. A geodesic with a deep cusp excursion is sparse because the
corresponding deck transformation admits a parabolic factorization.
If the geodesic stays below a prescribed cusp-depth cutoff, a
self-intersection splits it into two arcs; subdividing the longer arc
into a bounded number of pieces gives a product of lattice elements
satisfying the intersection polynomial. The escape estimate from the
preceding section then yields an exponential saving. Balancing these two
contributions will give the gap in \cref{thm:main}.

We continue under the orientability assumption made above and treat the
nonorientable case in \cref{rem:nonorientable-count}. For a
closed geodesic $\gamma$, define its depth
$$
\operatorname{dep}(\gamma)
=\max_{x\in\gamma}\operatorname{dep}(x),
$$
where the pointwise depth is the one fixed in
\eqref{eq:cusp-depth}. If $M$ is compact,
$\operatorname{dep}(\gamma)$ is $0$.

For later use, if $\operatorname{dep}(x)\le H$, then
\eqref{eq:coarse-orbit-density} gives
\begin{equation}\label{eq:truncated-orbit-density}
d(x,\Gamma o)\le K_H,
\tn{ where }
K_H=H+D_{\mathrm{thick}}.
\end{equation}

\begin{lemma}
\label{lem:deep-cusp-count}
For $L,H\ge1$, we have 
$$
\#\{\gamma:\gamma\text{ primitive},\ \ell(\gamma)\le L,
\ \operatorname{dep}(\gamma)>H\}
\ll_M e^{(d-1)(L-H)}.
$$
\end{lemma}

\begin{proof}
We may assume that $M$ has a cusp. The idea is to separate
a long excursion inside a horoball from the complementary part of the
closed geodesic. After normalizing the entrance and exit points on the
boundary horosphere, these two pieces give a parabolic factor and a
residual factor, respectively.

Fix such a primitive closed geodesic of length $\leq L$ and choose a cusp excursion
reaching depth $H'>H$. Let $r$ be the length of the part of
the excursion inside the corresponding horoball. Normalize the upper half-space model so that the boundary horosphere
has Euclidean height $1$. The lifted excursion is an arc of a Euclidean
semicircle whose highest point has height $e^{H'}$. By symmetry, the hyperbolic distance from either endpoint of the
excursion to its highest point is $r/2$. The standard distance formula
therefore gives
$\cosh(r/2)=e^{H'}$.
Consequently, 
\begin{equation}\label{eq:deep-excursion-length}
r=2\arccosh(e^{H'})\ge2H.
\end{equation}

There are finitely many cusp orbits. For each representative cusp,
fix a compact fundamental set $\mathcal E_j$ on its boundary
horosphere.
Choose a lift of the excursion and conjugate the primitive deck
transformation so that the entrance point $x$ belongs to some
$\mathcal E_j$.
If $y$ is the exit point, choose $p$ in the corresponding maximal
parabolic subgroup so that $p^{-1}y$ lies in $\mathcal E_j$. Orient the lifted axis so that the excursion runs from $x$ to $y$, and
let $g\in\Gamma$ be the primitive isometry translating in this
direction. With this choice, $y\in[x,gx]$, and
$$
d(x,gx)=\ell(g),\qquad d(x,y)=r, \quad d(y,gx)=\ell(g)-r. 
$$ 
Since $x$ and $p^{-1}y$ lie in a fixed compact set and $M$ has only finitely many cusps, 
$$
d(y,px)=O_M(1).
$$
Since $x$ lies in a fixed compact fundamental set $\mathcal{E}_j$, 
$$
\begin{aligned}
d(o,po)
&\le d(o,x)+d(x,y)+d(y,px)+d(px,po)\\
&\le O_M(1)+r+ O_M(1)+ O_M(1)=r+ O_M(1). 
\end{aligned}
$$
Put
$$
b=p^{-1}g.
$$ 
Then
$$
\begin{aligned}
d(o,bo)
&\le d(o,p^{-1}y)
   +d(p^{-1}y,p^{-1}gx)
   +d(p^{-1}gx,p^{-1}go)\\
&=O_M(1)+\ell(g)-r+O_M(1).
\end{aligned}
$$
Now fix $n\ge0$ and suppose that $r\in[n,n+1)$. The parabolic factor
$p$ therefore satisfies
$$
d(o,po)\le n+O_M(1). 
$$
Thus \eqref{eq:parabolic-ball-count} gives
$$
\#\{p\in P:d(o,po)\le n+O_M(1)\}
\ll_M e^{(d-1)n/2}.
$$
Similarly,
$$
d(o,bo)\le L-n+O_M(1),
$$
and hence \eqref{eq:lattice-ball-growth} gives
$$
\#\{b\in\Gamma:d(o,bo)\le L-n+O_M(1)\}
\ll_M e^{(d-1)(L-n)}.
$$
Thus the number of pairs $(p,b)$ arising from excursions with
$r\in[n,n+1)$ is
$$
\ll_M
e^{(d-1)n/2}e^{(d-1)(L-n)}
=
e^{(d-1)(L-n/2)}.
$$

For each closed geodesic with length at most $L$ and depth greater than $H$, we choose one normalized excursion at the $j$-th cusp neighborhood and its corresponding pair $(p,b)$. The data $(j,p,b)$ determine $g=pb$, and
hence the primitive geometric image. Thus the resulting assignment is injective.
By \eqref{eq:deep-excursion-length}, $n\ge2H-1$. Summing the preceding
estimate over these values of $n$ establishes the lemma.
\end{proof}

\begin{lemma}
\label{lem:shallow-nonsimple-count}
For $L,H\ge1$, we have
\begin{equation}\label{eq:shallow-nonsimple-count}
\begin{aligned}
&\#\left\{\gamma:
\begin{array}{l}
\gamma\text{ primitive and nonsimple},\quad \ell(\gamma)\le L,\\
\operatorname{dep}(\gamma)\le H
\end{array}
\right\}\\
&\hspace{2cm}\ll_M
\exp\left(
\left(d-1-\frac{\eta_{\mathcal P}}2\right)L
+2(d-1)(m_0+1)H
\right).
\end{aligned}
\end{equation}
\end{lemma}

\begin{proof}
Fix a primitive nonsimple closed geodesic $\gamma$ of length
$\ell\le L$ and depth $\le H$. Choose a self-intersection
$q\in\gamma$ and an arc $\alpha\subset\gamma$ whose endpoints are both
$q$ and whose length $t$ satisfies
$$ 
0<t\le\frac{\ell}{2}.
$$ 
Let $g\in\Gamma$ be the primitive deck transformation corresponding to
this orientation of $\gamma$. Choose a unit-speed parametrization
$\widetilde\gamma\colon\mathbb R\to A_g$ with
$x=\widetilde\gamma(0)$ projecting to $q$. Since
$\widetilde\gamma(t)$ also projects to $q$, there is an element
$h\in\Gamma$ such that
$$
\widetilde\gamma(t)=hx,
\qquad
\widetilde\gamma(\ell)=gx.
$$
Since the primitive elements of $\Stab_\Gamma(A_g)$ have translation
length $\ell$, while
$$
d(x,hx)=t\le\frac{\ell}{2},
$$
we have $h\notin\Stab_\Gamma(A_g)$. Moreover,
$
hx$ lies in $A_g\cap hA_g.$
By conjugating $g$ and $h$ simultaneously and using
\eqref{eq:truncated-orbit-density}, we may assume that
$$
d(o,x)\le K_H.
$$
Put
$$
c=h^{-1}g,
\qquad
s=\ell-t\ge t.
$$
Thus the isometry $h$ records the shorter branch between the two occurrences of the
self-intersection, while $c$ records the complementary branch. We count
$h$ by the ambient lattice growth and apply polynomial escape only to
the longer factor $c$. Then
\begin{equation}\label{eq:shallow-two-arcs}
d(x,hx)=t,
\qquad
d(x,cx)=s,
\qquad
[x,cx]=h^{-1}[hx,gx].
\end{equation}
The segments $[x,hx]$ and $[x,cx]$ project to subarcs of $\gamma$ and
therefore have depth at most $H$. Furthermore,
\begin{equation}\label{eq:u-displacement-bound}
d(o,ho)\le d(o,x)+d(x,hx)+d(hx,ho)\le t+2K_H.
\end{equation}
By \cref{lem:collision-polynomial}, we have the vanishing of the polynomial
\begin{equation}\label{eq:shallow-collision-polynomial}
P_h(g)=P_h(hc)=0.
\end{equation}
Let
$$
n=\lfloor t\rfloor.
$$
Then
$$
n\le t<n+1,
\qquad
0\le n\le \frac{L}{2},
$$
and, since $s=\ell-t$ and $\ell\le L$,
$$
s\le L-n.
$$
To apply \cref{cor:ball-product-escape}, divide the longer branch
$[x,cx]$ into $m_0$ pieces. We will approximate each piece by an isometry $c_i$ to be defined below. 
Choose points $z_0=x,z_1,\ldots,z_{m_0}=cx$ on $[x,cx]$ so that
$$
d(z_{i-1},z_i)
\le
\frac{L-n}{m_0}.
$$
Set $a_0=e$ and $a_{m_0}=c$. By
\eqref{eq:truncated-orbit-density}, there exist $a_i\in\Gamma$ for
$1\le i<m_0$ so that
$$
d(a_io,z_i)\le K_H.
$$
The same bound holds for $i=0,m_0$ because $d(o,x)\le K_H$ and
$d(co,cx)=d(o,x)$. Thus, writing
$$
c_i=a_{i-1}^{-1}a_i,
$$
we obtain the factorization
\begin{equation}\label{eq:shallow-exact-factorization}
c=c_1\cdots c_{m_0},
\tn{ where for each }c_i,\  
d(o,c_io)
\le
\frac{L-n}{m_0}+2K_H.
\end{equation}

For fixed $h\ne e$, define
$$
F_h(z)=P_h(hz).
$$
Since left translation by $h$ preserves the class of polynomial
functions on $G$,
the polynomial $F_h$ does not vanish identically on $G$ and has the same uniform degree
bound. Hence $F_h$ belongs to $\mathcal P\setminus\{0\}$.

Apply \eqref{eq:ball-product-count} with
$$
R=L-n,
\qquad
R_i=\frac{L-n}{m_0}+2K_H
\quad (1\le i\le m_0).
$$
The number of tuples $(c_1,\ldots,c_{m_0})$ satisfying
\eqref{eq:shallow-collision-polynomial} and
\eqref{eq:shallow-exact-factorization} is at most
\begin{equation}\label{eq:shallow-tuple-count}
\begin{aligned}
C_Me^{-\eta_{\mathcal P}(L-n)}
\prod_{i=1}^{m_0}
\left|B_{(L-n)/m_0+2K_H}\right|
\le
C_M
\exp\left(
(d-1-\eta_{\mathcal P})(L-n)
+2(d-1)m_0K_H
\right).
\end{aligned}
\end{equation}
From $t<n+1$, \eqref{eq:u-displacement-bound} gives $d(o,ho)<n+1+2K_H$. Hence the number of possible $h$ is at most
$$
C_Me^{(d-1)(n+2K_H)}, 
$$
after absorbing the factor $e^{d-1}$ into the implicit constant.

For each geodesic, fix one choice of the self-intersection, the associated
elements $g,h\in\Gamma$, and one factorization
\eqref{eq:shallow-exact-factorization}. The resulting
tuple $(h,c_1,\ldots,c_{m_0})$ determines
$$
g=hc_1\cdots c_{m_0},
$$
and therefore determines the primitive geometric image. Thus the
resulting assignment is injective.

For fixed $n$, the preceding estimates give
$$
\ll_M
\exp\left(
(d-1-\eta_{\mathcal P})L
+\eta_{\mathcal P}n
+2(d-1)(m_0+1)K_H
\right).
$$
Since $0\le n\le L/2$,
$$
\sum_{0\le n\le L/2}e^{\eta_{\mathcal P}n}
\ll
e^{\eta_{\mathcal P}L/2}.
$$
Summing over $n$ and using
$K_H=H+D_{\mathrm{thick}}$ therefore gives
$$
\ll_M
\exp\left(
\left(d-1-\frac{\eta_{\mathcal P}}2\right)L
+2(d-1)(m_0+1)H
\right),
$$
which proves \eqref{eq:shallow-nonsimple-count}.
\end{proof}

\begin{remark}[The nonorientable case]\label{rem:nonorientable-count}
Suppose that $M$ is nonorientable. Write
$$
\Gamma=\Gamma^+\sqcup\Gamma^+\xi
$$
for a fixed $\xi\in\Gamma\setminus\Gamma^+$.
The incidence argument of \cref{lem:collision-polynomial} applies to
both components of $\operatorname{Isom}(\mathbb H^d)$, with the same
uniform degree bound.

To extend \cref{lem:shallow-nonsimple-count} to the nonorientable case, first note that $\Gamma^+$ has index two in $\Gamma$. Hence, after adjusting the constant in
\eqref{eq:truncated-orbit-density}, the intermediate orbit points $a_i$ in the proof of \cref{lem:shallow-nonsimple-count} may be chosen in $\Gamma^+$. Write the last factor as
$c_{m_0}=y_{m_0}\xi^\varepsilon$, where $
y_{m_0}\in\Gamma^+$ and $\varepsilon\in\{0,1\}$. Thus we have 
$$
c=c_1\cdots c_{m_0-1}y_{m_0}\xi^\varepsilon,
\tn{ where }
c_1,\ldots,c_{m_0-1},y_{m_0}\in\Gamma^+,
$$
and
$$
d(o,y_{m_0}o)
\le
\frac{L-n}{m_0}+2K_H+O_M(1).
$$
For each $\varepsilon\in\{0,1\}$, composing the corresponding
intersection polynomial with the algebraic map
$$
z\longmapsto hz\xi^\varepsilon
$$
gives a nonzero polynomial $F_{h,\varepsilon}\in\mathcal P$ of the
same uniformly bounded degree. We may therefore apply
\cref{cor:ball-product-escape} to the orientation double cover
$\Gamma^+\backslash\mathbb H^d$ and to the factors
$c_1,\ldots,c_{m_0-1},y_{m_0}\in\Gamma^+$. This yields the estimate in
\cref{lem:shallow-nonsimple-count}, with only the implicit constant
changed. The proof of \cref{lem:deep-cusp-count} applies to the full
group $\Gamma$ without modification. Hence the preceding
bounds remain valid without the orientability assumption.
\end{remark}

\begin{proof}[Proof of \cref{thm:main}]
By \cref{rem:nonorientable-count}, both
\cref{lem:deep-cusp-count,lem:shallow-nonsimple-count} hold without
an orientability assumption.
We choose the truncation depth to balance the deep and shallow estimates. Put
$$
\alpha=
\frac{\eta_{\mathcal P}}
{2(d-1)(2m_0+3)},
\qquad
\delta_d=(d-1)\alpha
=
\frac{\eta_{\mathcal P}}{2(2m_0+3)},
$$
and set $H=\alpha L$. For all sufficiently large $L$, we have $H\ge1$,
so \cref{lem:deep-cusp-count,lem:shallow-nonsimple-count} give
$$
\begin{aligned}
N_M^{\mathrm{ns}}(L)
\ll_M
 e^{(d-1)(L-H)} +
\exp\left(
\left(d-1-\frac{\eta_{\mathcal P}}2\right)L
+2(d-1)(m_0+1)H
\right).
\end{aligned}
$$
By the choice of $\alpha$, we have two equalities 
$$
(d-1)H=\delta_dL
$$
and
$$
\frac{\eta_{\mathcal P}}2
-2(d-1)(m_0+1)\alpha
=\delta_d.
$$
Hence
$$
N_M^{\mathrm{ns}}(L)
\ll_M
e^{(d-1-\delta_d)L}.
$$
Since $m_0=\dim\mathcal P-1$ and the space $\mathcal P$ is determined
by $d$ and the degree bound $D_0(d)$, the constant $\delta_d$ depends
only on $d$. 
The prime geodesic theorem, in the form stated in the introduction,
implies
$$
N_M(L)\asymp_M\frac{e^{(d-1)L}}{L}.
$$
Consequently, we have 
$$
\frac{N_M^{\mathrm{ns}}(L)}{N_M(L)}
\ll_M
Le^{-\delta_dL}
\longrightarrow0,
$$
and therefore
$$
\frac{N_M^{\mathrm{simp}}(L)}{N_M(L)}\longrightarrow1.
$$
\end{proof}


\section{Nonsimple closed geodesics in arithmetic manifolds of type~I}\label{sec:dense-nonsimple}
Recall that Chinburg--Reid \cite{CR} construct closed hyperbolic $3$-manifolds all of whose closed geodesics are simple. Our goal here is to prove \Cref{thm:dense-nonsimple}, which states that, for every arithmetic manifold of type~I, the unit tangent vectors to nonsimple closed geodesics are dense in the unit tangent bundle $T^1M$. It is well known that an arithmetic hyperbolic manifold of type~I contains infinitely many noncommensurable finite-volume totally geodesic hypersurfaces. 
\begin{proof}
For $d=2$, it follows from \cite{bsGeodesicsBoundedIntersection} that the union of unit vectors tangent to nonsimple closed geodesics in a finite-area complete hyperbolic surface $S$ is dense in $T^1S$. 

We first assume $d=3$.  
    Let $M$ be an arithmetic hyperbolic $3$-manifold of type~I. Then there exists a sequence of finite-area hyperbolic surfaces $F_i$ and totally geodesic immersions 
    $$
\iota_i\colon F_i\to M
$$
whose images $S_i=\iota_i(F_i)$ are pairwise noncommensurable.  
   For simplicity of notation, we denote the pushforward $(\iota_i)_*$ of $\gt F_i$ by $\gt S_i$. If $M$ is noncompact, then each $S_i$ intersects the compact core of $M$ (whose complement is a union of cusp neighborhoods). 

     Let $v,w\in T^1_xM$ be a pair of orthogonal unit vectors at $x$. Then $(x,v,w)$ determines a tangential two-plane. By Mozes--Shah \cite{msErgodicInvariantMeasure}, the Hausdorff limit of $\gt S_i $ is $\gt M$. Thus there exist $(x_i, \Pi_i)\in \gt S_i$ converging to $(x,\Pi)$. Consequently, for any $\epsilon>0$, there exists $(x_i, v_i)\in T^1S_i$, such that 
     $$ d((x_i, v_i), (x, v)) \leq \epsilon. $$
     Birman--Series imply that there exists $(x_i', v_i')\in T^1S_i$ which is tangent to nonsimple closed geodesics on $S_i$ such that $$d((x_i', v_i'), (x_i, v_i)) \leq \epsilon.$$
    Since $\epsilon$ is arbitrary, this completes the proof when the dimension is three. 
    
Assume that $M$ is an arithmetic hyperbolic $d$-manifold of type I $(d\geq 4)$ and the theorem holds for such manifolds of dimension $\leq d-1$. Then $M$ contains a sequence of pairwise noncommensurable finite-volume totally geodesic hypersurfaces $S_i$ which are also arithmetic of type I. By Mozes--Shah again, the Hausdorff limit of $\mg_{d-1} S_i$ is $\mg_{d-1} M$. On each $S_i$, unit vectors tangent to nonsimple closed geodesics are dense by the induction hypothesis. Thus, by a similar argument, for any $(x,v)\in T^1M$,
there exist unit vectors $(x_i,v_i)$ tangent to nonsimple closed geodesics $\gamma_i\subset M$ so that $(x_i,v_i)$ converge to $(x,v)$. 
\end{proof}
Also relying on rigidity of unipotent flows due to Ratner \cite{rmRatner'stheorem}, Mozes, and Shah, [\citenum{hxhNearlyGeodesicFilling}, Theorem 1.6] proves that in a closed hyperbolic $d$-manifold ($d\geq 3$), all but finitely many commensurability classes of closed totally geodesic hypersurfaces are filling (cutting the ambient manifold into topological balls), thus answering the question posed by Wu--Xue \cite{wxPrimeGeodesicTheorem} in higher dimensions. 

 \paragraph{\textbf{Further directions}.}   
The argument also suggests extensions beyond lattices in real
hyperbolic space. Natural settings include convex cocompact Kleinian
groups and finite-volume quotients of the other rank-one symmetric
spaces of noncompact type. The product-escape mechanism is largely
algebraic; the main additional issue is to construct an appropriate
algebraic encoding of intersecting geodesics. 

\section*{Acknowledgments}
We are grateful to Yunhui Wu for bringing the Chinburg--Reid question
to our attention and for sharing many insights on random hyperbolic
geometry. We thank Nathan Dunfield for stimulating conversations about
random $3$-manifolds and the work in \cite{dtRandom3-manifolds}. The idea of using
random models to study the existence of infinitely many simple closed
geodesics arose while XHH was randomly walking at
the Tsinghua Sanya International Mathematics Forum during the program
``Geometric Topology and Related Topics.'' He thanks the organizers
and the center for their hospitality.

Alan Reid pointed out an example of a hyperbolic manifold containing
nonsimple closed geodesics but no closed totally geodesic surfaces;
this example ruled out an approach we had been considering. We also
thank Yitwah Cheung and Jinxin Xue for helpful discussions on
dynamical systems, and Nick Miller for explaining the hierarchy of
totally geodesic submanifolds in arithmetic manifolds used in the
induction in \cref{thm:dense-nonsimple}.

We used ChatGPT and Google Gemini to assist with reference searches,
background checks, computational verification, and language editing.
The authors take full responsibility for the content of the
manuscript.

XHH is partially supported by a start-up grant from the Shanghai
Institute for Mathematics and Interdisciplinary Sciences and by NSFC
grant No.\ 12501084. BY is supported by a start-up grant from the
Shanghai Institute for Mathematics and Interdisciplinary Sciences.

\bibliographystyle{alpha}
\bibliography{ref}
\end{document}